\documentclass[10pt]{amsart}
\usepackage{amsmath,amssymb,verbatim}
\usepackage[utf8]{inputenc}
\usepackage[english]{babel}
\usepackage{enumerate}
\usepackage{pdflscape}
\usepackage{amsthm}
\usepackage{mathrsfs}
\usepackage{color}
\usepackage{multicol}
\usepackage[normalem]{ulem}
\usepackage{cancel}
\usepackage{tikz}
\usetikzlibrary{matrix}
\usepackage[all]{xy}
\usepackage{mathtools}
\usepackage{arydshln}
\usepackage[shortlabels]{enumitem}
 \usepackage[english]{babel}
\usepackage{hyperref}
\hypersetup{
    colorlinks = true,
    linkbordercolor = {white},
    allcolors = blue
} 
\usepackage[capitalize, noabbrev]{cleveref}
 
\makeatletter
\@namedef{subjclassname@1991}{\textup{2020} Mathematics Subject Classification}
\makeatother

\topskip=\baselineskip 
\newtheorem{theorem}{Theorem}[section]
\newtheorem{lemma}[theorem]{Lemma}

\newtheorem{example}[theorem]{Example}

\newtheorem{corollary}[theorem]{Corollary}
\newtheorem{remark}[theorem]{Remark}

\newtheorem*{theorem*}{Theorem}

\newtheorem{maintheorem}{Theorem}

\makeatletter
\@addtoreset{equation}{section}
\makeatother

\newcommand{\Z}{{\mathbb Z}}

\newcommand{\GEN}[1]{\left\langle #1 \right\rangle}

\newcommand{\ZZ}{\mathrm{Z}}

\newcommand{\VSd}{V(S_0)}
\newcommand{\VSdd}{V}

\newcommand{\qand}{\quad \text{and} \quad}

\title[On the modular isomorphism problem for groups with center of index at most $p^3$]{On the modular isomorphism problem  \\ for groups with center of index at most $p^3$}
 
 \author{Sofia   Brenner}
\address{Department  of  Mathematics,   TU  Darmstadt,   Germany.}
\email{brenner@mathematik.tu-darmstadt.de}
 
 \author{Diego García-Lucas}
 \address{Departamento de Matem\'aticas, Universidad de Murcia, Spain}
 \email{diego.garcial@um.es}

\thanks{The research of the first author has received funding from the European Research Council (ERC) under
	the European Unions Horizon 2020 research and innovation programme (EngageS: grant agreement No.~820148). The research of the second author was partially supported by Grant PID2020-113206GB-I00 funded by MCIN/AEI/10.13039/501100011033  and by Grant Fundación Séneca 22004/PI/22.}



\date{\today}

\begin{document}
	
	\maketitle
	
	\begin{abstract}
	Let $p$ be an odd prime number. We show that the modular isomorphism problem has a positive answer for finite $p$-groups whose center has index $p^3$, which is a strong contrast to the analogous situation for $p = 2$.
	\end{abstract}

\section{Introduction} 

Let $p$ be a prime number and let $F$ be the field with $p$ elements. In this paper, we study the modular isomorphism problem, which can be stated as follows: \medskip

\emph{Consider finite $p$-groups $G$ and $H$ such that the group algebras $FG$ and $FH$ are isomorphic. Is it true that $G$ and $H$ are isomorphic?}
\medskip


In the past 50 years, this problem has seen extensive research. There are many classes of finite $p$-groups for which it is known to have a positive answer, such as the class of abelian  $p$-groups \cite{Deskins1956}, the class of $p$-groups  of nilpotency class $2$ with elementary abelian derived subgroup \cite{San89}, or the class of metacyclic groups \cite{BaginskiMetacyclic, San96}.
  The modular isomorphism problem has a positive answer for groups with a center of index $p^2$. This follows from the earlier-mentioned result by Sandling \cite{San89}. Moreover, the generalization of the problem for this class of groups to arbitrary fields of characteristic $p$ has a positive solution, as shown by    \cite{Drensky}. The modular isomorphism problem is also known to have positive answer for groups of order dividing $p^5$ for any prime \cite{Passman1965p4, SalimSandlingp5}, for groups of order dividing $2^8$, for groups of order dividing $3^7$ and, except for a few groups, for groups of order dividing $5^6$ \cite{Eick08, MM20}. 
 Some recent results are for example \cite{BK19,MS22}, and a complete overview can be found in the recent survey~\cite{Mar22}. 
\medskip

Despite all positive results, it was shown in \cite{GarciaMargolisdelRio} that the general answer to this problem is negative. The smallest counterexample consists of a pair of groups of order $2^9$ with centers of index $8$.  For $p > 2$, the modular isomorphism problem is still open.
\medskip

 For a $p$-group $G$, we write $\ZZ(G)$ for its center, $\Phi(G)$ for its Frattini subgroup and $(\gamma_i(G))_{i \geq 1}$ for its lower central series. Motivated by the counterexample of \cite{GarciaMargolisdelRio} for $p = 2$, we now study finite $p$-groups with a center of index $p^3$ for $p > 2$. In contrast to the situation for $p = 2$, it turns out that the modular isomorphism problem has a positive answer for this class of groups. Our main result is the following: 
 
\begin{maintheorem}\label{theo:maintheorema}
	Let $p>2$ be a prime integer, let $F$ be the field with $p$ elements, and let $G$ and $H$ be  finite $p$-groups. Suppose that $|G:\ZZ(G)|=p^3$. If $FG\cong FH$, then $G\cong  H$. 
\end{maintheorem} 

Note that the class of $p$-groups $G$ with $|G:\ZZ(G)| = p^3$ is not contained in the union of the classes of $p$-groups for which the modular isomorphism problem is known to have a positive answer. Corresponding examples are discussed in \Cref{sec:indexp3}. 
The key ingredient for the proof of \Cref{theo:maintheorema} is the following result, which might also be of independent interest:

\begin{maintheorem}\label{theo:maintheoremb}                                                                                
Let $p>2$ be a prime integer, let $F$ be the field with $p$ elements, and consider finite $p$-groups~$G$ and $H$ with $F G\cong F H$. Let $d \in \Z_{\geq 0}$ with $|G/\Phi(G)\ZZ(G)|=p^d$. If 
\begin{enumerate}
		\item   $G^p\cap \gamma_2(G)\subseteq \gamma_2(G)^p \gamma_3(G)$ and 
		\item  $|\gamma_2(G)/\gamma_2(G)^p\gamma_3(G)|=p^{{d\choose 2}} ,$
	\end{enumerate}
	then $G/\gamma_2(G)^p \gamma_4(G)\cong  H/\gamma_2(H)^p \gamma_4(H)$.
\end{maintheorem}

This statement is a generalization of \cite[Theorem~1.2]{MM20} when $p>2$, as for $d=2$, it is equivalent to the following implication: if $ G/\ZZ(G )$ is $2$-generated, then the isomorphism class of $G/\gamma_2(G)^p\gamma_4(G)$ is determined by $FG$.
\medskip

This paper is organized in the following way: In \Cref{sec:preliminaries}, we introduce the notation used in this paper. In \Cref{sec:smallgroupalgebra}, we derive results on the so-called small group algebra, which will lead to the proof of~\Cref{theo:maintheoremb}. Finally, in \Cref{sec:indexp3}, we prove \Cref{theo:maintheorema}.

\section{Preliminaries and notation}\label{sec:preliminaries}
Let $p$ be an odd prime number and let $G$ be a finite $p$-group. We use the standard group-theoretic notation. In particular, $Z(G)$ and $\Phi(G)$ denote the center and the Frattini subgroup of $G$, respectively. A \emph{Burnside basis} is a generating system of $G$ of minimal size. 
For $a,b \in G$, let $[a,b] = a^{-1} b^{-1} a b$. Moreover, we set $\gamma_1(G) = G$ and $\gamma_{i+1}(G) = [\gamma_i(G), G]$ for $i\geq 1$. 

\begin{lemma}\label{lemma:computation}
	Let $G$ be a finite $p$-group of nilpotency class at most $3$. Then for each $x,y\in G$, 
	$$[x^p,y] = [x,y]^p [[x,y] ,x]^{{p\choose 2}} \qand [x,y^p]=[x,y]^p [[x,y],y]^{{p\choose 2}}.$$
\end{lemma}
\begin{proof}
	Using several times the identity $$[xz,y]=[x,y]^z [z,y]=[x,y][[x,y],z] [z,y], \qquad ( x,y,z\in G) $$   we derive that	\begin{align*}
	[x^p,y]&=  [x,y]^{x^{p-1}}\cdot [x,y]^{x^{p-2}}\cdots [x,y] \\
	&=[x,y] [[x,y],x]^{p-1}\cdot [x,y]^{x^{p-2}}[[x,y],x]^{p-2}\cdots [x,y] \\
	&= [x,y]^p [[x,y],x]^{ {p \choose 2}}.
	\end{align*}   
	A similar argument shows the second identity. 
\end{proof}
As an immediate consequence, we obtain:
\begin{lemma}\label{lemma:AgemoCentral}
	Let $G$ be a finite $p$-group with $p>2$. If $\gamma_2(G)^p\gamma_4(G)=1$, then $G^p\subseteq \ZZ(G)$.
\end{lemma}

 Let $F$ be a field of characteristic $p$. In this paper we only consider unitary, finite-dimensional $F$-algebras.
 An \emph{augmented} $F$-algebra is an $F$-algebra $A$ endowed with an $F$-algebra homomorphism  $A\to F$, called the \emph{augmentation map}. The group algebra $FG$ is an augmented $F$-algebra with augmentation map $\sum_{g\in G} a_gg\mapsto  \sum_{g\in G} a_g$. The kernel $I(A)$ of the augmentation map is called the \emph{augmentation ideal} of~$A$. For $A=FG$, we abbreviate $I(FG)$ by $I(G)$.

 Let $A$ be an augmented local $F$-algebra, as it is the case for $A=FG$. Then we denote $V(A)= 1+I(A)$, the subgroup of the group of units of $A$ formed by the elements of augmentation $1$.  Every quotient of $A$ by some ideal $J$ of $A$ is also an augmented algebra with an augmentation map induced by the augmentation of~$A$. We can (and will)  make the identifications $I(A/J)=I(A)/J$ and $V(A/J)=1+I(A/J)=V(A)/(1+J)$.
We define the set $K(A)$ to be the $F$-subspace of $A$ generated by the  elements of the form $ab-ba$ with  $a,b\in A$.

	
%

Note that $I(G) I(\gamma_2(G))$ is a $2$-sided ideal of $FG$. A main tool in this paper will be the investigation of the \emph{small group algebra} $S = FG/I(G) I(\gamma_2(G))$. This object has been extensively studied by R.~Sandling and M.~A.~M.~Salim (see \cite{San89, SalimSandling1995, SalimTesis}).  It was used in~\cite{SalimSandlingp5} to positively answer the modular isomorphism problem for groups of order dividing $p^5$. Moreover, the analogous construction for   group rings over the integers was used in  A.~Whitcomb's thesis \cite{Whitcomb} to positively answer the isomorphism problem  for integral group rings for metabelian groups.   Further results about the small group algebra in a wider sense can be found in   \cite{HertweckSoriano07}.  
By a bar, we denote both the natural projection $\overline\cdot :FG\to S$, which is an $F$-algebra homomorphism, and its restriction $\overline \cdot :V(FG)\to V(S)$, which can be viewed as a group homomorphism. A subgroup of $V(FG)$ that is an $F$-basis for $FG$ is called a \emph{group basis} of $FG$. 
We say a property $P$ is \emph{determined} by the group algebra if for any pair of $p$-groups $G$ and $H$ such that $FG\cong FH$, the group $G$ satisfies $P$ if and only if $H$ satisfies $P$.

We have that $(1+I(G)I(\gamma_2(G)))\cap G=\Phi(\gamma_2(G))$   (see the introduction of \cite{SalimSandling1995}), so $\overline{G} \cong G/\Phi(\gamma_2(G))$. Note that this holds for any group basis of $FG$. However, it has been known for a long time that the information encoded by the small group algebra alone is not enough to recover the isomorphism type of $G/\Phi(\gamma_2(G))$ (see \cite[Section 4]{Bag99}), even for groups of nilpotency class $3$ (see \cite[Example 3.11]{MS22} and the subsequent discussion).

We frequently use the following result due to Sandling:  

\begin{lemma}[{\cite{San89}, see also \cite[Proposition~3.2]{MM20}}]\label{lemma:definitiona}
	Let $G$ be a finite $p$-group. 
	Let $g_1, \ldots, g_m \in \overline{G}$ be elements whose images in $\overline{G}/\gamma_2(\overline{G})$ form an independent generating set as an abelian group. Let $A$ be the subgroup of $V(S)$ generated by the elements $1 + (g_1-1)^{k_1} \cdots (g_m-1)^{k_m}$ with $k_1, \ldots, k_m \in \Z_{\geq 0}$ and $k_1 + \dots + k_m \geq 2$. Then $V(S) = \overline{G}A$. For $a = 1 + (g_1-1)^{k_1} \cdots (g_m-1)^{k_m} \in A$ and $g \in \overline{G}$, we have \[[g,a] = [g, g_1, g_1, \dots, g_1, g_2, \dots, g_2, \dots, g_m, \dots, g_m],\] where each $g_i$ appears exactly $k_i$ times. If additionally $\gamma_4(G) = 1$ holds, then $A$ is abelian and $V(S) = \overline{G} \rtimes A$. 
\end{lemma}

Furthermore, \cite[Theorem~1.1]{SalimSandling1995} yields that 
\begin{equation}\label{eq:gammaivs}
\gamma_i(V(S))=\gamma_i(\overline{G})
\end{equation} 
for each $i\geq 2$. 

\section{Results on the small group algebra}\label{sec:smallgroupalgebra}
Throughout, let $F$ denote the field with $p$ elements for a prime number $p> 2$ and let $G$ be a finite $p$-group. In this section, we derive results on the small group algebra that will later be used to solve the modular isomorphism problem for groups with a center of index~$p^3$. 

We use the notation introduced in the preceding section. In particular, let $\overline{\cdot} \colon FG\to S$ be the natural projection onto the small group algebra $S$ and let $V(S) = 1 + I(G)/I(G) I(\gamma_2(G))$ denote the group of units of augmentation $1$ of $S$. Furthermore, let $A$ be a subgroup of $V(S)$ as described in Lemma~\ref{lemma:definitiona}. 


\begin{remark}
For every normal subgroup $N$ of $G$ contained in $\gamma_2(G)$, it follows easily  from the identity $xy-1=x-1+y-1+(x-1)(y-1)$  that \[I(N)FG+ I(\gamma_2(G))I(G)= N-1+I(\gamma_2(G))I(G).\] 
\end{remark}

In particular, $\gamma_i(\overline G)-1 = \overline{\gamma_i(G)-1}$ is an ideal in $S$.
In the following, we set $\Gamma = \gamma_4(G) -1 \subseteq FG$ and consider the quotient algebra  $S_0 = FG/(I(G) I(\gamma_2(G)) + \Gamma)  \cong S/\overline{\Gamma}$ together with the corresponding projection $\pi_0\colon FG \to S_0$.
Then $\pi_0(G)= G/\gamma_2(G)^p \gamma_4(G)$, and the same holds for every group basis of~$FG$. 
Furthermore, we obtain $\VSd = \pi_0(G) \rtimes A$ (see \Cref{lemma:definitiona}). 
More generally, the following holds:

 \begin{lemma}\cite[Proposition 3.1]{MS22}\label{lemma:complement}
 	Let $H$ be any group basis of $FG$. Then $\VSd=\pi_0(H)\rtimes A$.
 \end{lemma}

Finally, we consider the natural projection from $\VSd$ onto $\VSdd= \VSd/ (A\cap \ZZ(\VSd))$. Composing it with the restriction of~$\pi_0$ to $V(FG)$ yields a group homomorphism 
\begin{equation}\label{eq:defpi}
\pi \colon V(FG)\to \VSdd.
\end{equation} Hence $\VSdd =\pi(G)\rtimes E$ follows, where $E=A/ (A\cap \ZZ(\VSd))$ is elementary abelian due to $A^p\subseteq \ZZ(\VSd)$ (see \Cref{lemma:definitiona}). Therefore, by \Cref{lemma:complement}, we have 
\begin{equation}\label{eq:decompv}
	V = \pi(H) \rtimes E
\end{equation}
  for any group basis $H$ of $FG$. 

In the next lemma, we derive a presentation for groups satisfying the hypothesis of Theorem~\ref{theo:maintheoremb}. 

 \begin{lemma}\label{lemma:presentation}  
	Write $|G/\ZZ(G)\Phi(G)|=p^d$, $|\gamma_2(G)^p \gamma_3(G)/\gamma_2(G)^p\gamma_4(G)|=p^m$ and $|G/\Phi(G)|=p^k$. If \begin{enumerate}[(a)]
		\item  $G^p\cap \gamma_2(G)\subseteq \gamma_2(G)^p \gamma_3(G)$ and 
		\item $|\gamma_2(G)/\gamma_2(G)^p\gamma_3(G)|=p^{{d\choose 2}} $,
	\end{enumerate}  then there exist positive integers $n_i$  and  $0\leq \alpha_{is}, \beta_{ijls}\leq p-1$ for   $i,j,l\in \{1,\dots, k\}$ and $s\in \{1,\dots, m\}$ such that 
$$G/\gamma_2(G)^p \gamma_4(G)\cong \GEN{g_1,\dots, g_k ,c_1,\dots, c_m \mid \mathcal R_{ (n_i,\alpha_{ij},\beta_{ijls})} ( g_1,\dots, g_k,c_1,\dots, c_m)   },$$   where $\mathcal R_{( n_i,\alpha_{ij},\beta_{ijls})} ( g_1,\dots, g_k,c_1,\dots, c_m) $ is the set formed by the following relations: 
\begin{enumerate}
	\item \label{1} $g_i^{p^{n_i}}=\prod_{1\leq j \leq m}c_j^{\alpha_{ij}}$   for each   $1\leq i  \leq k$.
	\item \label{2}$ [[g_i,g_j], g_l]=\prod_{1\leq s\leq m} c_s^{\beta_{ijls}}$ for $1\leq i,j,l\leq d$.
	\item\label{3} $[g_i,g_j]^p=1$ for  $1\leq i ,j\leq d$.
	\item\label{4} $g_i$ is central for $d<i\leq k$. 
	\item\label{5} The nilpotency class of the group is at most $3$.
\end{enumerate} 
Moreover, we can choose the elements $g_{d+1},\dots ,g_{k} \in G/\gamma_2(G)^p \gamma_4(G)$ such that they have central preimages in $G$.
\end{lemma}
\begin{proof}
Set $\widetilde{G}=G/\gamma_2(G)^p \gamma_4(G)$ and, for $g \in G$, write $\widetilde{g} = g \gamma_2(G)^p \gamma_4(G) \in \widetilde{G}$ (similarly for subsets of~$G$). Observe that $|\gamma_3(\widetilde{G})|=p^m$ and $|\widetilde{G}/\Phi(\widetilde{G})|=p^k$. Write $ |\widetilde{G}/\ZZ(\widetilde{G}) \Phi(\widetilde{G})|=p^t$. As the image of $Z(G) \Phi(G)$ in $\widetilde{G}$ has index at most~$p^d$, we obtain   
$p^t\leq p^d$.  Moreover, $\gamma_2(\widetilde G)$ is elementary abelian and the hypotheses imply that $ \widetilde {G}^p \cap\gamma_2(\widetilde G)\subseteq \gamma_3(\widetilde G)  $ and $|\gamma_2(\widetilde G)/\gamma_3(\widetilde G)|=p^{d\choose 2}$.  Since ${\widetilde G}^p\subseteq \ZZ(\widetilde{G})$ by \cref{lemma:AgemoCentral}, we have a   well-defined map $\mathfrak c \colon \widetilde{G}/\ZZ(\widetilde{G}) \Phi(\widetilde{G}) \times \widetilde{G}/\ZZ(\widetilde{G}) \Phi(\widetilde{G}) \to  \gamma_2(\widetilde G)/\gamma_3(\widetilde G) $ induced by the group commutator map, which is bilinear and anti-symmetric. If $\{x_1,\dots, x_t\}$ is a basis of $\widetilde{G}/\ZZ(\widetilde{G}) \Phi(\widetilde{G})$, then $\mathfrak c(x_i,x_j)$ with $1\leq i <j \leq t$ generates $\gamma_2(\widetilde G)/\gamma_3(\widetilde G)$ as an $F$-vector space. Thus $p^{{d\choose 2}}=|\gamma_2(\widetilde G)/\gamma_3(\widetilde G)|   \leq p^{{t\choose 2}}$. This yields $t=d$.

Let $\{x_1,\dots, x_d\}$ be a set of preimages of a basis of $  G/\ZZ(  G)\Phi(  G)$ in $G$, and let $\{x_{d+1},\dots x_k$\} be a set of central preimages in $G$ of a basis of $\ZZ(G)\Phi(G)/\Phi(G)$. Thus $\{x_1,\dots, x_k\}$ is a Burnside basis of $G$, and its image is a Burnside basis of $\widetilde G$. By abuse of notation, we regard  $\{x_1,\dots, x_k\}$ as a subset of $\widetilde G$. Let $\{z_1,\dots , z_m\}$ be a basis of $\gamma_3(\widetilde G)$. Write $p^{n_i}$ for the order of $x_i \gamma_2(\widetilde{G})$ in $\widetilde{G}/\gamma_2(\widetilde G)$. As $\widetilde {G}^p\cap \gamma_2(\widetilde G) \subseteq \gamma_3(\widetilde G)$, we have that $x_i^{p^{n_i}}=\prod_{1\leq j \leq m} z_j^{\alpha_{ij}}$ for uniquely determined $\alpha_{ij} \in \{0, \dots, p-1\}$. Moreover, for $1\leq i,j,l\leq d$, we can write $ [[x_i,x_j], x_l]=\prod_{1\leq s\leq m} z_s^{\beta_{ijls}}$ for uniquely determined $\beta_{ijls} \in \{0, \dots, p-1\}$. As $\{z_1, \ldots, z_m\}$ is a basis of $\gamma_3(\widetilde G)$, the matrix $(\beta_{ijls})_{ (i,j,l), s}$,   with rows indexed by triples $(i,j,l)$ and columns indexed by~$s$, has rank $m$. 
	
	Let \[H=\GEN{g_1,\dots, g_k,c_1,\dots ,c_m \mid \mathcal R_{ (n_i,\alpha_{ij},\beta_{ijls})} ( g_1,\dots, g_k,c_1,\dots, c_m) }.\] 	 The assignment $g_i \mapsto x_i$ ($i = 1, \ldots, k$) and $c_j \mapsto z_j$ ($j = 1, \ldots, m$) defines a surjective group homomorphism $H\to \widetilde G$. To complete the proof it suffices to show that it is an isomorphism, or equivalently that $|H|\leq  |\widetilde G|$.

	 Observe that since the matrix $(\beta_{ijls})_{ (i,j,l)\times s}$  has rank $m$, relation \eqref{2} implies that $c_1, \ldots, c_m \in \gamma_3(H)$. 	In particular, we have $c_1, \ldots, c_m \in \Phi(H)$ and hence $|H/\Phi(H)|\leq p^k$.

	  As $H$ has nilpotency class at most $3$, $\gamma_2(H)$ is abelian.   The group $\gamma_2(H)$ is generated by $\gamma_3(H)$ and elements of order~$p$ by relation \eqref{3}, and hence $\gamma_2(H)/\gamma_3(H)$  is elementary abelian. Thus so is $\gamma_2(H)/\gamma_2(H)\cap \ZZ(H)$.   On the other hand, the image of the map $H/\ZZ(H)\times \gamma_2(H)/\gamma_2(H) \cap \ZZ(H)\to \gamma_3(H)$  
	  induced by the commutator map generates $\gamma_3(H)$. For $h \in H$ and $h' \in \gamma_2(H)$, we have  
	$[h \ZZ(H), h' \gamma_2(H)\cap \ZZ(H)]^p= [h\ZZ (H), (h')^p \gamma_2(H)\cap \ZZ(H)]=1$ (see \Cref{lemma:computation}). Thus $\gamma_3(H)$ is elementary abelian.
	This implies that $\gamma_2(H)$ is elementary abelian (as an abelian group generated by elements of order~$p$).  
	
 By relations \eqref{2} and \eqref{5}, we have $\gamma_3(H) = \langle c_1, \dots, c_m\rangle$, so
	\begin{equation}\label{eq:gamma3h}
 |\gamma_3(H)|\leq  p^m=|\gamma_3(\widetilde G)|.
	\end{equation} 
 Moreover, using \eqref{4}, we derive  that $|H/\ZZ(H)\Phi(H)|\leq p^d$. 
	Observe that $[H,\Phi(H)] \subseteq \gamma_3(H)$ due to $[x^p,y] \gamma_3(H) =  [x,y]^p \gamma_3(H)$ for all $x,y \in H$ (see Lemma~\ref{lemma:computation}) and $\gamma_2(H)$ being elementary abelian. Then the commutator in the group defines a bilinear map  $H/Z(H) \Phi(H) \times H/Z(H) \Phi(H)\to  \gamma_2(H)/\gamma_3(H)$, $(x\ZZ(H)\Phi(H), y\ZZ(H)\Phi(H))\mapsto [x,y] \gamma_3(H)$.
	As for $\widetilde{G}$, 
	we derive that
	\begin{equation}\label{eq:quotientgamma2gamma3}
	|\gamma_2(H)/\gamma_3(H)|\leq p^{{d\choose 2}}=|\gamma_2(\widetilde G)/\gamma_3(\widetilde G)|. 
	\end{equation}
	 Finally, $H/\gamma_2(H)=\GEN{g_1 \gamma_2(H),\dots, g_k \gamma_2(H)}$ has order at most $\prod_{i=1 }^k p^{n_i}=|\widetilde G/\gamma_2(\widetilde G)|$ by relation \eqref{1}. Since $ |\widetilde G|= |\widetilde G/\gamma_2(\widetilde G)|\cdot |\gamma_2(\widetilde G)/\gamma_3(\widetilde G) |\cdot  |\gamma_3(\widetilde G)| $, combining this with \eqref{eq:gamma3h} and~\eqref{eq:quotientgamma2gamma3} yields that $|H|\leq |\widetilde G|$. This proves the statement. 
\end{proof}





\begin{theorem}\label{theorem:1}  Let $p$ be an odd prime number and let $F$ be the field with $p$ elements. Consider finite $p$-groups $G$ and $H$ and           
	           let $|G/\Phi(G)\ZZ(G)|=p^d$. If 
	           \begin{enumerate}[(a)]
		\item\label{lemma:presentation1}   $G^p\cap \gamma_2(G)\subseteq \gamma_2(G)^p \gamma_3(G)$ and 
		\item\label{lemma:presentation2}  $|\gamma_2(G)/\gamma_2(G)^p\gamma_3(G)|=p^{{d\choose 2}} ,$
	\end{enumerate}
then $G/\gamma_2(G)^p \gamma_4(G)\cong  H/\gamma_2(H)^p \gamma_4(H)$.
\end{theorem}
\begin{proof}
Consider the isomorphism of elementary abelian $p$-groups 
$$\varphi \colon G/\Phi(G)\to I(G)/I(G)^2, \quad x\Phi(G)\mapsto (x-1) + I(G)^2$$
(see~\cite[Propositions~III.1.14 and~III.1.15]{Seh78}). 
We have an $F$-vector space decomposition $\ZZ(FG) = F\ZZ(G) + \ZZ(FG) \cap K(FG)$ (see \cite[Lemma 6.10]{Sandling85}). It is well-known that $FG \cdot K(FG) = K(FG) \cdot FG = I(\gamma_2(G)) FG$. As $\gamma_2(G) \subseteq \Phi(G) = (1 + I(G)^2) \cap G$ holds (see \cite[Theorem~5.5]{Jen41}),
 we have 
 \begin{equation}\label{eq:kg}
K(FG) \subseteq I(G') FG \subseteq I(G)^2.
 \end{equation} In particular, $\varphi$ restricts to an isomorphism
\begin{equation}\label{eq:isocenters}
\ZZ(G)\Phi(G)/\Phi(G) \to  \left( \ZZ(FG)\cap I(G)+I(G)^2 \right)/I(G)^2.
\end{equation}	
Of course, the analogous reasoning applies to $H$. In particular, we obtain
$$p^d=|G/\Phi(G)\ZZ(G)|=|I(G)/(\ZZ(FG) \cap I(G) + I(G)^2)|= |H/\Phi(H)\ZZ(H)|.$$
Moreover, note that the conditions \eqref{lemma:presentation1} and \eqref{lemma:presentation2} depend only on the quotient $G/\gamma_2(G)^p \gamma_3(G)$, which is determined by $FG$ by \cite{San89}. Thus, if $G$ satisfies \eqref{lemma:presentation1} and \eqref{lemma:presentation2}, then the analogous statements hold for $H$. 
 
Consider the group $V$ and the projection $\pi \colon V(FG) \to V$ introduced in \eqref{eq:defpi}. Let $\{g_1,g_2,\dots, g_k, c_1,\dots,c_m\}$ be a set of generators of $\pi(G)$ satisfying relations $$\mathcal R_{( n_i,\alpha_{ij},\beta_{ijls})} ( g_1,\dots, g_k,c_1,\dots, c_m) $$ as in \Cref{lemma:presentation}. For $1\leq i \leq k$, the decomposition $\VSdd =\pi(H)\rtimes E$ (see \eqref{eq:decompv}) guarantees that $g_i=h_i e_i$ for some unique $h_i\in \pi(H)$ and $e_i\in E$. Moreover $c_j \in \gamma_2(\pi(G))=\gamma_2(S)=\gamma_2(\pi(H))$  for $j=1,\dots, m$  (see~\eqref{eq:gammaivs}). Now let $d<i \leq k$, so $g_i \in Z(\pi(G))$. We have 
\[(g_i -1) + I(G)^2 = (h_i e_i-1)+ I(G)^2 = h_i -1 + e_i - 1 + I(G)^2 = (h_i-1) + I(G)^2,\] 
using  $e_i -1 \in I(G)^2$ (see \cref{lemma:definitiona}).  
In particular, this yields $h_i -1 + I(G)^2 \in (Z(FG) \cap I(G) + I(G)^2)/I(G)^2$. Using the isomorphism in \eqref{eq:isocenters}, and the analogue for $H$, we derive that $h_i =\hat h_i \widetilde h_i$, with $\hat h_i\in \ZZ(\pi(H))
$ and $\widetilde h_i\in\Phi (\pi( H)) = \gamma_2(\pi(H)) \pi(H)^p$. Observe that $\pi(H)^p$ is central in $\pi(H)$ by \Cref{lemma:AgemoCentral}, so we can furthermore assume that $\widetilde h_i \in \gamma_2(\pi(H))$.
Then $\{h_1,\dots, h_d, \hat h_{d+1}, \dots , \hat h_k,c_1,\dots, c_m\}$ forms a generating set of $\pi(H)$. We now show that it satisfies the relations $$\mathcal R_{ (n_i,\alpha_{ij},\beta_{ijls})} ( h_1,\dots, h_d, \hat h_{d+1}, \dots , \hat h_k,c_1,\dots, c_m). $$  Then~$H$ is an epimorphic image of $G$ of the same size and we obtain $G\cong H$ as desired. 
	  
	  We check the relations one by one. Relations \eqref{3}, \eqref{4} and \eqref{5} are immediate.   
	  For $1\leq i \leq k$, one has $$g_i^p=(h_ie_i)^p=h_i^pe_i^p [h_i,e_i]^p=h_i^p,$$ and, in particular for $d<i\leq k$, one has $g_i^p=h_i^p= \hat h_i^p \widetilde h_i^p = \hat h_i ^p$. Then condition \eqref{1} follows. Condition \eqref{2} also follows readily because, for  $1\leq i \leq d$  and arbitrary $x_i,x_j,x_l\in  A$, we have that 
	  \begin{equation}
	  [[g_ix_i, g_jx_j], g_lx_l] =[[g_i,g_j] [x_j, g_j][g_i,x_j] , g_l x_l ] =[[g_i,g_j], g_l x_l]=[[g_i,g_j],g_l].
	\end{equation}
This completes the proof. 
\end{proof}

The next result shows that \Cref{theorem:1} is a generalization of \cite[Theorem 1.2]{MM20}. 
\begin{corollary}\label{corollary:1}
	Let $p$ be an odd prime number, and let $G$ and $H$ finite $p$-groups such that $F G\cong F H$. If $|G/\Phi(G)\ZZ(G)|=p^2$, then $$G/\gamma_2(G)^p \gamma_4(G)\cong H/\gamma_2(H)^p\gamma_4(H).$$ 
\end{corollary}
\begin{proof}
 If $G/ \gamma_2(G)^p \gamma_4(G)$ has nilpotency class $2$, then so does $H/\gamma_2(H)^p \gamma_4(H)$ (see \eqref{eq:gammaivs}) and hence the result follows by~\cite{San89}. Thus we assume otherwise. Observe that $\gamma_2(G)/\gamma_2(G)^p\gamma_3(G)$ has order $p$, as it is generated by $[g_1,g_2]\gamma_2(G)^p\gamma_3(G)$ if we consider a set of generators $\{g_1,\dots ,g_k,c_1,\dots, c_m\}$   as in \Cref{lemma:presentation}.  Observe that every $p$-power is central in $G/ \gamma_2(G)^p \gamma_4(G)$ since $[x,y^p] \equiv [x,y,y]^{{p\choose 2}} \equiv 1 \pmod{\gamma_2(H)^p \gamma_4(H)}$ for all $x,y \in G$ (by Lemma~\ref{lemma:computation} and using $p>2$). On the other hand, the image of $[g_1,g_2]$ in $G/\gamma_2(G)^p \gamma_4(G)$ is not central as~$G$ does not have nilpotency class $2$. This yields that $G^p \cap \gamma_2(G)\subseteq \gamma_2(G)^p \gamma_3(G)$. Thus the hypotheses of \Cref{theorem:1} with $d=2$ hold, and the statement follows. 
\end{proof}

\section{\texorpdfstring{Groups with a center of index $p^3$}{Groups with a center of index p^3}}\label{sec:indexp3}

Let $p$ be an odd prime number. In this section, we use the results about the small group algebra from the previous section to give a positive answer to the modular isomorphism problem for groups with center of index $p^3$. 
 
\begin{lemma}\label{lemma:descriptionp3}
	Let $p$ be an odd prime number and let $G$ be a  finite $p$-group with $|G:\ZZ(G)|=p^3$. Then:
	\begin{enumerate}
		\item The nilpotency class of $G$ is at most $3$.
		\item $\gamma_2(G)$ is elementary abelian. 
		\item Either $G$ has nilpotency class $2$ or $|G/\Phi(G)\ZZ(G)|=p^2$.
	\end{enumerate}
	
\end{lemma}
\begin{proof}
Examination of the upper central series yields that the nilpotency class of $G$ is at most 3. The quotient $G/\ZZ(G)$ is isomorphic to one of the five groups of order $p^3$, namely to one of $C_{p^3}$, $C_p^3$, $C_{p^2}\times C_p$,  $(C_p\times C_p)\rtimes C_p$ or $C_{p^2}\rtimes C_p$. Observe that $G/Z(G)$ is not cyclic as $G$ is abelian otherwise.
	
 	Next, we show that $\gamma_2(G)$ is elementary abelian. Then $\gamma_2(G)$ is abelian, as $[\gamma_2(G), \gamma_2(G)]\subseteq \gamma_4(G)=1$. Thus it suffices to show that the generators of $\gamma_2(G)$ have order $p$. Note that $\gamma_3(G)$ is elementary abelian.  Indeed, it is generated by  elements of the form $[x,y]$ with $x\in G$ and $y\in \gamma_2(G)$; since $\gamma_2(G)\ZZ(G)/\ZZ(G)$ has order at most $p$, we have $y^p\in \ZZ(G)$, so \Cref{lemma:computation} shows that $1=[x,y^p]=[x,y]^p $. 
 	Let  $a,b_1, b_2\in G$ be such that $G/\ZZ(G)=\GEN{a\ZZ(G), b_1 \ZZ(G),   b_2\ZZ(G)}$. By the structure of $G/\ZZ(G)$  we may assume that $ b_i^p\in \ZZ(G) $ for $i=1,2$. Then for every $x\in G$ and $i \in \{1,2\}$, \Cref{lemma:computation} yields that
	$$1=[x , b_i^p]= [x,b_i]^p [[x,b_i],b_i]^{{p\choose 2}} =[x,b_i]^p.$$
	This shows that $\gamma_2(G)$ is elementary abelian. 
	
	Finally, if $G$ has nilpotency class $3$, then $G/\ZZ(G)$ is isomorphic to one of the two non-abelian $p$-groups of order $p^3$. Then it is clear that $|\gamma_2(G)\ZZ(G)/\ZZ(G)|=p$, and hence $|G/\Phi(G)\ZZ(G)|=p^2$.
\end{proof} 

We can now prove our main result:
 
	\begin{theorem}\label{theorem:p3}
 Let $p$ be an odd prime number, let $F$ be the field with $p$ elements,  and 
		let $G$ and $H$ be  finite $p$-groups. Suppose that $|G:\ZZ(G)|=p^3$. If $FG\cong FH$, then $G\cong  H$. 
	\end{theorem} 
		 \begin{proof}
		 	By \Cref{lemma:descriptionp3}, we know that $G'$ is elementary abelian. If the nilpotency class of $G$ is $2$, then $G\cong H/\gamma_2(H)^p \gamma_3(H)$ by \cite{San89}, so $G\cong H$ due to $|G|=|H|$. Thus, by \Cref{lemma:descriptionp3}, we can assume that the nilpotency class of $G$ is $3$, and that $|G/\Phi(G)\ZZ(G)|=p^2$. By \Cref{corollary:1}, we then have $G\cong  H$.
		 \end{proof}
	 
	 \begin{remark}
We point out that there cannot be an analogue of \Cref{theorem:p3} for $p=2$, as in \cite{GarciaMargolisdelRio} non-isomorphic finite $2$-groups with centers of index $8$ and isomorphic group algebras over every field of characteristic $2$ are presented. Hence our result underlines the difference between the cases $p=2$ and $p>2$ for this problem (we refer to \cite{GLdRS2022} for other contrasts).  
	 \end{remark}
	 
%

We conclude this section with examples illustrating that the class of $p$-groups with a center of index $p^3$ is not contained in one of the classes for which the modular isomorphism problem was previously known to have a positive answer. 
%
%
\begin{example}
$\null$
\begin{enumerate} 
 	\item Consider the following groups of order $5^7$: 
	\begin{align*}
	G&=\GEN{a,b,c,d  \left| \begin{array}{c} 
			a^{625} = b^5=c^{25}=[b,a]^5=[c,a]=[c,b] =[d,a]=[d,b] =[d,c] =1, \\  a^5=d, 
			\left[  [b,a], a \right] = d^5, [[b,a],b]= c^5    	\end{array} \right. } \\ 
		&\\ 
		H&=    \GEN{a,b,c,d  \left| \begin{array}{c} 
				a^{625} = b^5=c^{25}=[b,a]^5=[c,a]=[c,b] =[d,a]=[d,b] =[d,c] =1, \\  a^5=d, 
				\left[  [b,a], a \right] = c^5, [[b,a],b]= d^5    	\end{array} \right. } \\
			&\\
			K&=	    \GEN{a,b,c,d  \left| \begin{array}{c} 
					a^{625} = b^5=c^{25}=[b,a]^5=[c,a]=[c,b] =[d,a]=[d,b] =[d,c] =1, \\  a^5=d^2, 
					\left[  [b,a], a \right] = c^5, [[b,a],b]= d^5    	\end{array} \right. }. 
	\end{align*} 
In {\sf GAP}, they are listed as ${ \rm \tt SmallGroup}(5^7, 1599)$, ${\rm \tt SmallGroup}(5^7,1734)$ and ${\rm \tt SmallGroup}(5^7,1766)$, respectively.
	 Each of these groups is $3$-generated with minimal set of generators $\{a,b,c\}$. The center is given by $ \langle c,d \rangle \cong C_{25} \times C_{25}$, and hence it has index $p^3$, so Theorem~\ref{theo:maintheorema} applies to $G,H$ and $K$. On the other hand, a quick verification shows that none of these groups is   covered by any of the cases for which MIP is known (according to the survey in \cite{Mar22}). Moreover, these three groups agree in all  the group theoretical invariants implemented in the {\sf GAP} package ModIsomExt~\cite{GAP, ModIsomExt}.

\item 	
Let $p \geq 7$. In the list of groups of order $p^6$ given in \cite{NEW23}, consider the groups $G_{6,3}$ as well as $G_{6,5r}$ for $r \in \{1, \nu\}$. These groups have centers of index $p^3$ and do not lie in the union of classes of $p$-groups for which the modular isomorphism problem is known to have a positive answer. For small values of~$p$, one can verify computationally that these groups agree in all group theoretical invariants implemented in ModIsomExt. 
\end{enumerate}
\end{example}

\section*{Acknowledgments} 
The authors are grateful to Leo Margolis and Ángel del Río for their comments on an earlier version of this paper.

 \bibliographystyle{amsalpha}
\bibliography{MIP}

\providecommand{\bysame}{\leavevmode\hbox to3em{\hrulefill}\thinspace}
\providecommand{\MR}{\relax\ifhmode\unskip\space\fi MR }
\providecommand{\MRhref}[2]{%
  \href{http://www.ams.org/mathscinet-getitem?mr=#1}{#2}
}
\providecommand{\href}[2]{#2}
\begin{thebibliography}{GLMdR22}

\bibitem[Bag88]{BaginskiMetacyclic}
C.~Bagi\'{n}ski, \emph{The isomorphism question for modular group algebras of
  metacyclic {$p$}-groups}, Proc. Amer. Math. Soc. \textbf{104} (1988), no.~1,
  39--42.

\bibitem[Bag99]{Bag99}
\bysame, \emph{On the isomorphism problem for modular group algebras of
  elementary abelian-by-cyclic {$p$}-groups}, Colloq. Math. \textbf{82} (1999),
  no.~1, 125--136.

\bibitem[BK19]{BK19}
C.~Bagi\'{n}ski and J.~Kurdics, \emph{The modular group algebras of
  {$p$}-groups of maximal class {II}}, Comm. Algebra \textbf{47} (2019), no.~2,
  761--771.

\bibitem[Des56]{Deskins1956}
W.~E. Deskins, \emph{Finite {A}belian groups with isomorphic group algebras},
  Duke Math. J. \textbf{23} (1956), 35--40.

\bibitem[Dre89]{Drensky}
V.~Drensky, \emph{The isomorphism problem for modular group algebras of groups
  with large centres}, Representation theory, group rings, and coding theory,
  Contemp. Math., vol.~93, Amer. Math. Soc., Providence, RI, 1989,
  pp.~145--153.

\bibitem[Eic08]{Eick08}
B.~Eick, \emph{Computing automorphism groups and testing isomorphisms for
  modular group algebras}, J. Algebra \textbf{320} (2008), no.~11, 3895--3910.

\bibitem[GAP19]{GAP}
The GAP~Group, \emph{{GAP -- Groups, Algorithms, and Programming, Version
  4.10.2}}, 2019, \url{https://www.gap-system.org}.

\bibitem[GLdRS22]{GLdRS2022}
D.~Garc\'{\i}a-Lucas, \'{A}. del R\'{\i}o, and M.~Stanojkovski, \emph{{On Group
  Invariants Determined by Modular Group Algebras: Even Versus Odd
  Characteristic}}, Algebr. Represent. Theory (2022),
  \url{https://doi.org/10.1007/s10468-022-10182-x}.

\bibitem[GLMdR22]{GarciaMargolisdelRio}
D.~García-Lucas, L.~Margolis, and Á. del Río, \emph{Non-isomorphic
  $2$-groups with isomorphic modular group algebras}, J. Reine Angew. Math.
  \textbf{154} (2022), no.~783, 269--274.

\bibitem[HS07]{HertweckSoriano07}
M.~Hertweck and M.~Soriano, \emph{Parametrization of central {F}rattini
  extensions and isomorphisms of small group rings}, Israel J. Math.
  \textbf{157} (2007), 63--102.

\bibitem[Jen41]{Jen41}
S.~A. Jennings, \emph{The structure of the group ring of a {$p$}-group over a
  modular field}, Trans. Amer. Math. Soc. \textbf{50} (1941), 175--185.

\bibitem[{Mar}22]{Mar22}
L.~{Margolis}, \emph{{The Modular Isomorphism Problem: A Survey}}, Jahresber.
  Dtsch. Math. Ver. \textbf{124} (2022), 157--196.

\bibitem[MM20]{ModIsomExt}
L.~Margolis and T.~Moede, \emph{{ModIsomExt}: a {GAP 4} package, version
  1.0.0}, 2020, \url{https://www. tu-braunschweig.de/en/iaa/personal/moede}.

\bibitem[MM22]{MM20}
L.~Margolis and T.~Moede, \emph{The {M}odular {I}somorphism {P}roblem for small
  groups – revisiting {E}ick's algorithm}, Journal of Computational Algebra
  \textbf{1-2} (2022), 100001.

\bibitem[MS22]{MS22}
L.~Margolis and M.~Stanojkovski, \emph{On the modular isomorphism problem for
  groups of class $3$ and obelisks}, J. Group Theory \textbf{25} (2022), no.~1,
  163--206.

\bibitem[NOVL]{NEW23}
M.~F. Newman, E.~A. O'Brien, and M.~R. Vaughan-Lee, \emph{Presentations for the
  groups of order $p^6$ for prime $p\geq 7$},
  \url{https://arxiv.org/abs/2302.02677}.

\bibitem[Pas65]{Passman1965p4}
D.~S. Passman, \emph{The group algebras of groups of order {$p^{4}$} over a
  modular field}, Michigan Math. J. \textbf{12} (1965), 405--415.

\bibitem[Sal93]{SalimTesis}
M.~A.~M. Salim, \emph{The isomorphism problem for the modular group algebras of
  groups of order $p^5$}, ProQuest LLC, Ann Arbor, MI, 1993, Thesis
  (Ph.D.)--University of Manchester.

\bibitem[San85]{Sandling85}
R.~Sandling, \emph{The isomorphism problem for group rings: a survey}, Orders
  and their applications ({O}berwolfach, 1984), Lecture Notes in Math., vol.
  1142, Springer, Berlin, 1985, pp.~256--288.

\bibitem[San89]{San89}
\bysame, \emph{The modular group algebra of a central-elementary-by-abelian
  {$p$}-group}, Arch. Math. (Basel) \textbf{52} (1989), no.~1, 22--27.

\bibitem[San96]{San96}
\bysame, \emph{The modular group algebra problem for metacyclic {$p$}-groups},
  Proc. Amer. Math. Soc. \textbf{124} (1996), no.~5, 1347--1350.

\bibitem[Seh78]{Seh78}
S.~K. Sehgal, \emph{Topics in group rings}, Monographs and Textbooks in Pure
  and Applied Math., vol.~50, Marcel Dekker, Inc., New York, 1978.

\bibitem[SS95]{SalimSandling1995}
M.~A.~M. Salim and R.~Sandling, \emph{The unit group of the modular small group
  algebra}, Math. J. Okayama Univ. \textbf{37} (1995), 15--25.

\bibitem[SS96]{SalimSandlingp5}
\bysame, \emph{The modular group algebra problem for groups of order {$p^5$}},
  J. Austral. Math. Soc. Ser. A \textbf{61} (1996), no.~2, 229--237.

\bibitem[Whi68]{Whitcomb}
A.~Whitcomb, \emph{The {G}roup {R}ing {P}roblem}, ProQuest LLC, Ann Arbor, MI,
  1968, Thesis (Ph.D.) -- The University of Chicago.

\end{thebibliography}

\end{document}